\documentclass[
twocolumn,
superscriptaddress,
 amsmath,amssymb,
 aps,
]{revtex4-2}

\usepackage{graphicx}
\usepackage{dcolumn}
\usepackage{bm}
\usepackage{todonotes}
\usepackage{comment}
\usepackage{appendix}
\usepackage{amsthm}
\usepackage{soul}
\usepackage{hyperref}

\newtheorem*{proposition*}{Proposition}

\newtheorem{lemma}{Lemma}[section]
\newtheorem{proposition}{Proposition}

\newcommand{\ri}{\mathrm{i}}

\DeclareMathOperator*{\argmin}{arg\,min}

\begin{document}

\title{Dynamically Optimal Projection onto Slow Spectral Manifolds for Linear Systems}

\author{Florian Kogelbauer}
\email{floriank@ethz.ch} 
\affiliation{Department of Mechanical and Process Engineering, ETH Z{u}rich, CH-8092 Z{u}rich, Switzerland}

\author{Ilya Karlin}
\email{ikarlin@ethz.ch}
\affiliation{Department of Mechanical and Process Engineering, ETH Z{u}rich,  CH-8092 Z{u}rich, Switzerland}

\date{\today}

\begin{abstract}
We derive the dynamically optimal projection onto the linear slow manifold from a temporal variational principle. We demonstrate that the projection captures transient dynamics of the overall dissipative system and leads to a considerably improved fit of reduced trajectories compared to full trajectories. We illustrate these optimal model reduction properties on explicit examples, including the linear three-component Grad's moment system. 
\end{abstract}

\maketitle

\section{Introduction}

Slow manifolds play a crucial role in model reduction, equilibrium analysis, and data assimilation \cite{gorban2018model}. Originating from the foundational works of Tikhonov \cite{tikhonov1948dependence} and Fenichel \cite{fenichel1979geometric} within geometric singular perturbation theory, slow manifolds have become essential tools in fields such as atmospheric sciences \cite{lorenz1992slow}, fluid dynamics \cite{holmes1997low}, and classical mechanics \cite{mackay2004slow}. The modern framework of slow spectral submanifolds, introduced in \cite{cabre2003parameterization,cabre2003parameterizationII,cabre2005parameterization}, builds on stable spectral properties near equilibria and quasi-periodic solutions. This theory was further advanced in \cite{haller2016nonlinear} and has recently found applications in data-driven nonlinear dynamics \cite{cenedese2022data}.

The central idea of slow-manifold reduction is based on the separation of time scales, arising from the ordering of negative real parts in the system's linearization. Similar concepts appear in the theory of inertial manifolds for dissipative PDEs \cite{robinson2003infinite}, which contain the global attractor and are defined by the long-term behavior of the dynamical system \cite{temam2012infinite}.

Once the existence of a slow manifold is established and an approximate parametrization is available, an important question arises: how do general initial conditions relate to the reduced dynamics on the slow manifold? While reduced dynamics on the slow manifold can still exhibit transients due to repeated eigenvalues with Jordan blocks, methods such as spectral submanifolds do not capture the transient dynamics associated with trajectories approaching the slow manifold from off-manifold initial conditions.

Even in globally stable systems trajectories converge to a unique minimum but may display prolonged transient behavior. In linear systems, such transients are linked to the non-normality of the governing operator. Transient stable dynamics are pervasive across physics and engineering, with examples from atmospheric science \cite{farrell1996generalized}, polymer dynamics \cite{ilg2000validity, ilg2002canonical}, and thermal convection \cite{john2012transient}. In fluid models, long transients have even been proposed as an explanation for turbulence \cite{baggett1995mostly,he2000study}, although this remains debated, given the inherently nonlinear nature of turbulence \cite{waleffe1995transition}. Transient dynamics are also believed to play a significant role in hypocoercivity estimates \cite{achleitner2025hypocoercivity}.

The problem of projecting onto the slow manifold in the presence of transient dynamics has a long history \cite{gear2005projecting}. General model reduction methods often lead to oblique projections that account for non-normality in linear settings \cite{antoulas2005approximation,gugercin2008h_2,rowley2005model}. Proper orthogonal decomposition combined with Galerkin approximations is one widely used technique \cite{rowley2004model} in this context. Near the slow manifold, the correct affine projections become position-dependent and satisfy a nonlinear partial differential equation \cite{roberts1989appropriate}, which can be approximated through series expansions via computer algebra \cite{roberts2000computer}. More recently, projections respecting appropriate fibers for reduced-order modeling have been developed using constrained autoencoders \cite{otto2023learning}.

In this work, we introduce a novel projection onto the slow manifold for linear systems, which we term the \textit{dynamically optimal projection}. This projection optimally captures the transient dynamics of general trajectories in phase space. Derived from a temporal variational principle, the dynamical projection reduces to the standard orthogonal projection in the case of normal operators, ensuring consistency with the classical Riesz projection \cite{hislop2012introduction} onto the span of eigenvectors. Unlike existing oblique projections that address non-normality, such as the $\mathcal{H}_2$-optimal projection \cite{gugercin2008h_2}, which compares full and reduced trajectories via their temporal Laplace transforms, our approach presumes knowledge of the slow manifold as a linear combination of eigenvectors associated with slow modes. This combines the rational of the slow spectral submanifold with global phase-space projections for linear stable systems. Additionally, our variational principle is formulated in Hilbert space, relying solely on the spectral properties of the operator generating the semigroup.
We also contrast the dynamically optimal projection with other projection methods from statistical physics. One notable example is the thermodynamic projector \cite{GORBAN2004391}, designed for stable systems with a global Lyapunov function and primarily used in kinetic and chemical dynamics. The thermodynamic projector requires only the gradient of the entropy, imparting a degree of universality, and can define a vector field on general, possibly non-invariant subsets of phase space. In contrast, the dynamically optimal projection requires a more detailed input, particularly the spectral data of the underlying linear operator.

\section{Spectral Theory and Slow Manifolds}
Consider an evolution equation of the form  
\begin{equation}\label{main}
    \frac{dx}{dt} = Lx,
\end{equation}
for a closed linear operator $L$ on the complex Hilbert space $H$ with inner product $\langle.,.\rangle$. For a linear operator $ L:H \to H$, we denote its spectrum as $\text{spec}(L)$, its adjoint as $L^\dagger$ and recall that $\text{spec}(L^\dagger) = \text{spec}(L)^*$, see \cite{conway2019course}. The commutator of two linear operators $L_1$ and $L_2$ is denoted as $[L_1,L_2]=L_1L_2-L_2L_1$.

We assume that the spectrum of $L$ is stable and consists of a continuous part plus a set of isolated eigenvalues of finite multiplicity above the essential spectrum, called the \textit{point spectrum}, i.e., 
\begin{equation}\label{propsigma}
    \text{spec}(L) = \text{spec}_{\rm ess}(L) \cup \{\lambda_j\}_{1\leq j\leq n},\ \Re\text{spec}(L)< 0,
\end{equation}
such that $\Re\text{spec}_{\rm ess}(L)< \Re\{\lambda_j\}_{1\leq j\leq n}$, where $n\in\mathbb{N}\cup \{\infty\}$. We assume that the eigenvalues are counted with multiplicity and that they ordered by real part such that $\Re\lambda_i\leq\Re\lambda_j$ for $i<j$. Furthermore, we assume that $L$ restricted to the eigenvectors associated to the point spectrum is diagonal, i.e., has a trivial Jordan-block structure. The case of repeated eigenvalues and generalized eigenvectors 
can be treated analogously and is omitted for notational simplicity. 

As mentioned before, operators satisfying \eqref{propsigma} appear in the context of fluctuations-dissipation theory as the Fokker--Planck operator of stochastically forced systems \cite{risken1996fokker} and kinetic theory \cite{cercignani1988boltzmann}. The case of a pure point spectrum, $\text{spec}_{\rm ess}(L)=\varnothing$, in particular finite-dimensional systems is also covered by  considerations below. 

The stability of the spectrum, $\Re \text{spec}(L)<0$, implies that the associated semi-group $e^{tL}$ is exponentially stable. In fact, the converse is also true \cite{engel2000one}. For kinetic systems, this corresponds to the increase of the quadratic Lyapunov function $x\mapsto -\langle x, x\rangle$, which is interpreted as the entropy of the system. In particular, any solution converges to zero as $t\to\infty$, but might exhibit large temporal transients before decaying. 

In the following, we will mostly be interested in the case when $L$ is non-normal, i.e., when $[L,L^\dagger]\neq 0$. For a treatment of the general spectral properties of non-normal operators derived as perturbations to multiplication operators, we refer to \cite{ljance1970completely}. This operator class is of particular interest in kinetic theory as regular perturbations to the transport operator \cite{kogelbauer2024rigorous}.

The eigenvectors $\{\hat{x}_j\}_{1\leq j \leq n}$ associated to the discrete eigenvalues, the slow modes compared to the essential spectrum, span an invariant hyperplane of the linear dynamical system \eqref{main}. Henceforth, we will call the hyperplane spanned by the eigenvectors the \textit{slow manifold} and denote it as $\mathcal{M}_{\rm slow}^n$. The separation of the slow manifold $\mathcal{M}_{\rm slow}^n$ from the dynamics associated to the essential spectrum is characterized by different time scales: The contributions to a general solution of the essential spectrum will decay faster as compared to the contributions of the point spectrum. 

Within the full slow manifold, the ordering of the eigenvalues by real part induces itself a hierarchy of time scales, i.e., attraction rates, and implies the existence of a nested family of invariant, slow manifolds,
\begin{equation}
    \mathcal{M}_{\rm slow}^1\subset  \mathcal{M}_{\rm slow}^2\subset ... \subset  \mathcal{M}_{\rm slow }^n. 
\end{equation}

The knowledge of the reduced model 
induced by operator $L$ on a slow manifold 
is only the first step, as one needs to know \textit{how} to project a general initial condition onto the slow manifold.  
As elaborated in the introduction, general solutions to \eqref{main} under the stability assumption \eqref{propsigma} may exhibit large transient dynamics for non-normal operators $L$, i.e., short- to medium-time growth properties and deviations from equilibrium before decaying exponentially fast. In general, reduced models 
are incapable of capturing these transient dynamics accurately.  Below, we will derive a projection of a generic initial condition onto the slow invariant manifold by minimizing the total-in-time deviation of a full trajectory to a solution on the slow manifold in a mean-square sense.

\section{Dynamically Optimal projection from a Variational Principle}

\begin{figure}
    \centering
\includegraphics[width=1\linewidth]{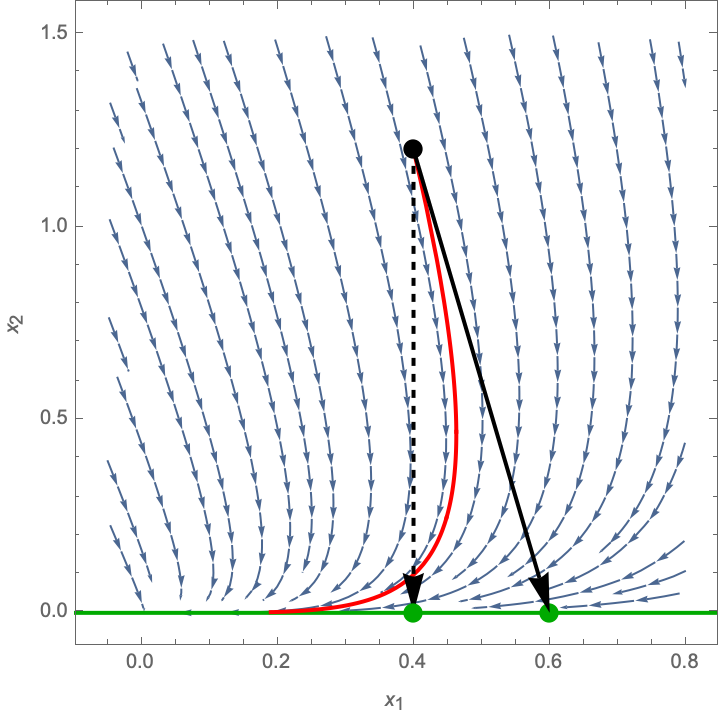}
    \caption{Illustration of the dynamical projection for system  \eqref{dyn2D} ($\alpha=5$, $\gamma=1$) applied to the initial condition $x_0=(0.4,1.2)$ for which the trajectory to the full system is shown in red. The  dynamically optimal projection (solid arrow) moves the initial condition on the slow manifold (x-axis in green) farther away from the global equilibrium compared to the orthogonal projection (dashed arrow).}
    \label{fig_proj_2D}
\end{figure}

Let $\{\hat{x}_j\}_{1\leq j\leq n} $ be a collection of slow modes spanning the slow manifold $\mathcal{M}_{\mathrm{slow}}^n$. The total dynamical error of a trajectory starting on the slow manifold compared to a general trajectory is given by,
\begin{equation}\label{E}
    \mathcal{E}(x_0,\xi ) = \frac{1}{2}\int_{0}^{\infty} \| e^{tL}x_0-e^{tL}  x_{\rm slow}(\xi )\|^2\, dt,
\end{equation}
for any initial condition $x_0\in H$ and a parameter vector $\xi \in\mathbb{C}^n$, where
\begin{equation}\label{slowmf}
    x_{\rm slow}(\xi ) = \sum_{j=1}^n \xi_{j} \hat{x}_j,
\end{equation}
is an initial condition on the slow manifold for the vector $\xi = (\xi_1,...,\xi_n)$. The reduced dynamics on the slow manifold for this parametrization is simply given by,
\begin{equation}\label{reddyn}
    \frac{d\xi_j}{dt} = \lambda_j \xi_j. 
\end{equation}
The question to be addressed here is: 
\emph{Which initial condition on the slow manifold corresponds to a generic initial condition of the system such that the error \eqref{E} will be minimal?}

Expression \eqref{E} is a well-defined, convex quadratic form in $\xi $, see Appendix \ref{propertiesE}, and hence has a unique minimum,
\begin{equation}\label{eq:min_problem}
   \xi^{\rm min}(x_0) := \argmin_{\xi \in\mathbb{C}^n} \mathcal{E}(x_0,\xi ). 
\end{equation}
Let us define the $n\times n$ spectrally-weighted Gramian matrix $G$ with the elements,
\begin{equation}\label{GX}
    G_{ij} = \frac{\langle\hat{x}_i ,\hat{x}_j\rangle}{\lambda_i+\lambda_j^*},\ 1\leq i,j \leq n,
\end{equation}
and mention its Hermitian symmetry, $G^\dagger=G$. Moreover, matrix $G$ is negative semi-definite: $\xi \cdot G\xi^*\le 0$, for $\xi\in\mathbb{C}^n$, see Appendix \ref{propertiesE}.

%
With the spectrally-weighted Gramian matrix \eqref{GX}, the explicit solution to the minimization problem \eqref{eq:min_problem} is found as follows: 
\begin{proposition}\label{prop:minimizer}
    The unique minimizer \eqref{eq:min_problem} reads,
\begin{equation}\label{betamin}
    \xi^{\rm min}_i(x_0) = \sum_{j=1}^n \left[\left(G^{T}\right)^{-1}\right]_{ij} \langle (L+\lambda_j^*)^{-1}x_0,\hat{x}_j \rangle,
\end{equation}
for $1\le i\le n$. 
\end{proposition}
\noindent The proof of Proposition \ref{prop:minimizer} is given in Appendix \ref{minexplicit}. 
Note that, the minimizer \eqref{betamin} is well-defined thanks to the stability assumption on the spectrum, $-\lambda^*\notin \text{spec}(L)$.

With the explicit form of the minimizer $\xi^{\rm min}$  \eqref{betamin}, we can define a linear operator
on $H$ as follows:
 \begin{equation}\label{defPdyn}
    \mathbb{P}_{\rm DOP} x =  \sum_{i=1}^n\sum_{j=1}^n\hat{x}_i\left[\left(G^{T}\right)^{-1}\right]_{ij} \langle (L+\lambda_j^*)^{-1}x,\hat{x}_j \rangle.
\end{equation}
Moreover, we have the following:
\begin{proposition}\label{prop:projector}
Linear operator $\mathbb{P}_{\rm DOP}$ \eqref{defPdyn} is a projection,
\begin{equation}\label{eq:PP=P}
\mathbb{P}_{\rm DOP}^2=\mathbb{P}_{\rm DOP}.
\end{equation}
\end{proposition}
\noindent The projection property \eqref{eq:PP=P} is proved in Appendix \ref{projpropapp}.

We call  $\mathbb{P}_{\rm DOP}$ \eqref{defPdyn}  the \textit{dynamically optimal projection} (DOP) associated to the slow manifold.
Clearly, the range of $\mathbb{P}_{\rm DOP}$ is the whole slow manifold \eqref{slowmf} by construction, while its kernel is given by,
\begin{equation}
\ker\mathbb{P}_{\rm DOP} = \{x\in H: (L+\lambda_j^*)^{-1}x\perp \hat{x}_{j},\ 1\leq j\leq n\}. 
\end{equation}
Let us summarize the properties of DOP: Given a general initial condition $x_0$ to system \eqref{main}, the dynamical projection \eqref{defPdyn} returns the unique initial condition $\mathbb{P}_{\rm DOP}x_0$ on the slow manifold that minimizes the cumulative difference of the full trajectory to its counterpart on the slow manifold. Intuitively, the solution of the optimization problem induces a projection by the minimality of the cumulative error of a full trajectory compared to a point of the parametrization on the slow manifold. Indeed, if the trajectory starts with an initial condition on the slow manifold, then the cumulative error \eqref{E} vanishes.

Let us take a closer look at the two ingredients of $\mathbb{P}_{\rm DOP}$. First, we emphasize that for a generic closed operator $L$, the spectrally-weighted Gramian \eqref{GX} does not simplify further since eigenvectors corresponding to different eigenvalues need not be orthogonal to each other. However, if $L$ is normal, we have, $L^\dagger\hat{x}_j = \lambda^*_j \hat{x}_j$, and eigenvectors corresponding to different eigenvalues are orthogonal \cite{conway2019course}. In this case, the spectrally-weighted Gramian ${G}$  \eqref{GX} reduces to a diagonal matrix, while the dynamically optimal projection becomes the canonical orthogonal projection, see Appendix \ref{Lnormal}.

Finally, let us comment that, for a general, not necessarily normal operator $L$, the projection $\mathbb{P}_{\rm DOP}$ \eqref{defPdyn} does not commute with the operator $L$, since the vectors $\hat{x}_i$ are not necessarily the eigenvectors of the adjoint operator $L^\dagger$. This distinguishes the DOP from the classical Riesz projection \cite{hislop2012introduction} which was recently used to project the (non-normal) linearized kinetic Boltzmann operator onto  hydrodynamic modes in \cite{kogelbauer2024rigorous}. Indeed, the Riesz projection onto a hyperplane spanned by eigenvectors of $L$ can be characterized as the unique projection that commutes with the operator $L$. A discussion of the Riesz projection through commutativity is provided in Appendix \ref{charRiesz}. The dynamically optimal projection can thus be viewed as an alternative to the Riesz projection which differs from it for non-normal operators $L$.

\begin{figure}
    \centering
\includegraphics[width=1\linewidth]{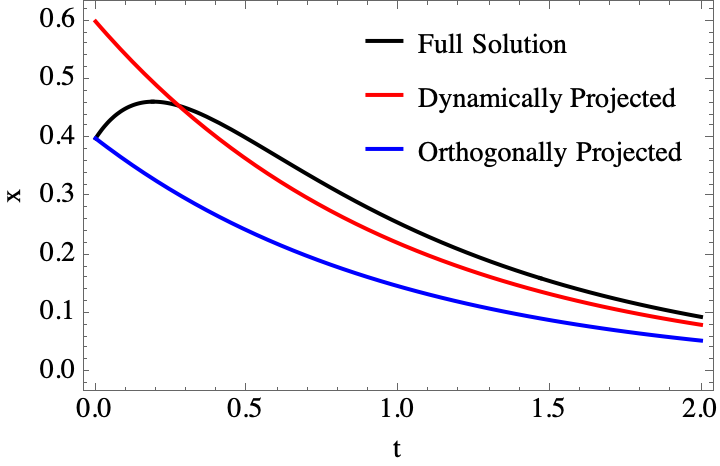}
    \caption{Comparison of the $x$-components of the full solution (black) to the dynamically projected trajectory (red) and the orthogonally projected trajectory (blue) of system \eqref{dyn2D}.}
    \label{fig_proj_2D_comp}
\end{figure}

\section{Examples}

In this section, we discuss two examples of the dynamic projection: a two-dimensional shear flow and the linear three-component Grad's moment system. For these examples, the spectrum is known analytically and the dynamic projection can be calculated explicitly. 

\subsection{Two-Dimensional Shear System}

As a first toy example, consider the two-dimensional linear system,
\begin{equation}\label{dyn2D}
    \frac{dx}{dt} = Lx,\quad L = \begin{pmatrix}
        -1 & \gamma \\
        0 & -\alpha 
        \end{pmatrix},
\end{equation}
where {$x=(x_1,x_2)^T$}, for a damping parameter $\alpha>1$ and a shear parameter $\gamma>0$. The eigenvectors of the non-normal matrix $L$ are,
\begin{equation}
    \text{eig}(L )=\left\{\begin{pmatrix}
        1 \\ 0
    \end{pmatrix}, \frac{1}{\sqrt{(1-\alpha)^2 + \gamma^2}}\begin{pmatrix}
        \gamma \\ 1-\alpha 
    \end{pmatrix} \right\},
\end{equation}
corresponding to the eigenvalues $-1$ and $-\alpha$, respectively. The slow manifold, associated to the eigenvalue $-1$, is therefore simply given by the $x_1$-axis. The resolvent appearing in \eqref{defPdyn} reads
\begin{equation}
    (L-1)^{-1} = \frac{(-1)}{2(1+\alpha)}\begin{pmatrix}
        1+\alpha & \gamma \\
        0 & 2 
    \end{pmatrix},
\end{equation}
while the spectrally-weighted Gramian \eqref{GX} reduces to a constant, $G = -1/2$. Thus, the dynamic projection \eqref{defPdyn} takes the explicit form,
\begin{equation}\label{dynpro2D}
    \mathbb{P}_{\rm DOP}x =  \left(x_1 + \frac{\gamma}{1 + \alpha}x_2\right)\begin{pmatrix}
        1 \\ 0
    \end{pmatrix}.
\end{equation}
For $\gamma=0$, the projection \eqref{dynpro2D} reduces to the standard orthogonal projection as expected. The dynamical projection for system \eqref{dyn2D} is shown in Figure \ref{fig_proj_2D}. Figure \ref{fig_proj_2D_comp} shows a direct comparison of the $x_1$-component of a full trajectory of system \eqref{dyn2D} to the dynamically projected and the orthogonally projected trajectories. We emphasize that the dynamically projected trajectory follows the overall trend towards equilibrium of the full solution much more closely over the full range of the time evolution. 
Geometrically, the functional \eqref{E} for this two-dimensional example seeks an initial condition on the slow manifold such that the weighted area enclosed by a full solution and the projected solution is minimal.

\subsection{Three-Component Grad's Moment System}
\label{GradExpl}

\begin{figure}
    \centering
    \includegraphics[width=0.9\linewidth]{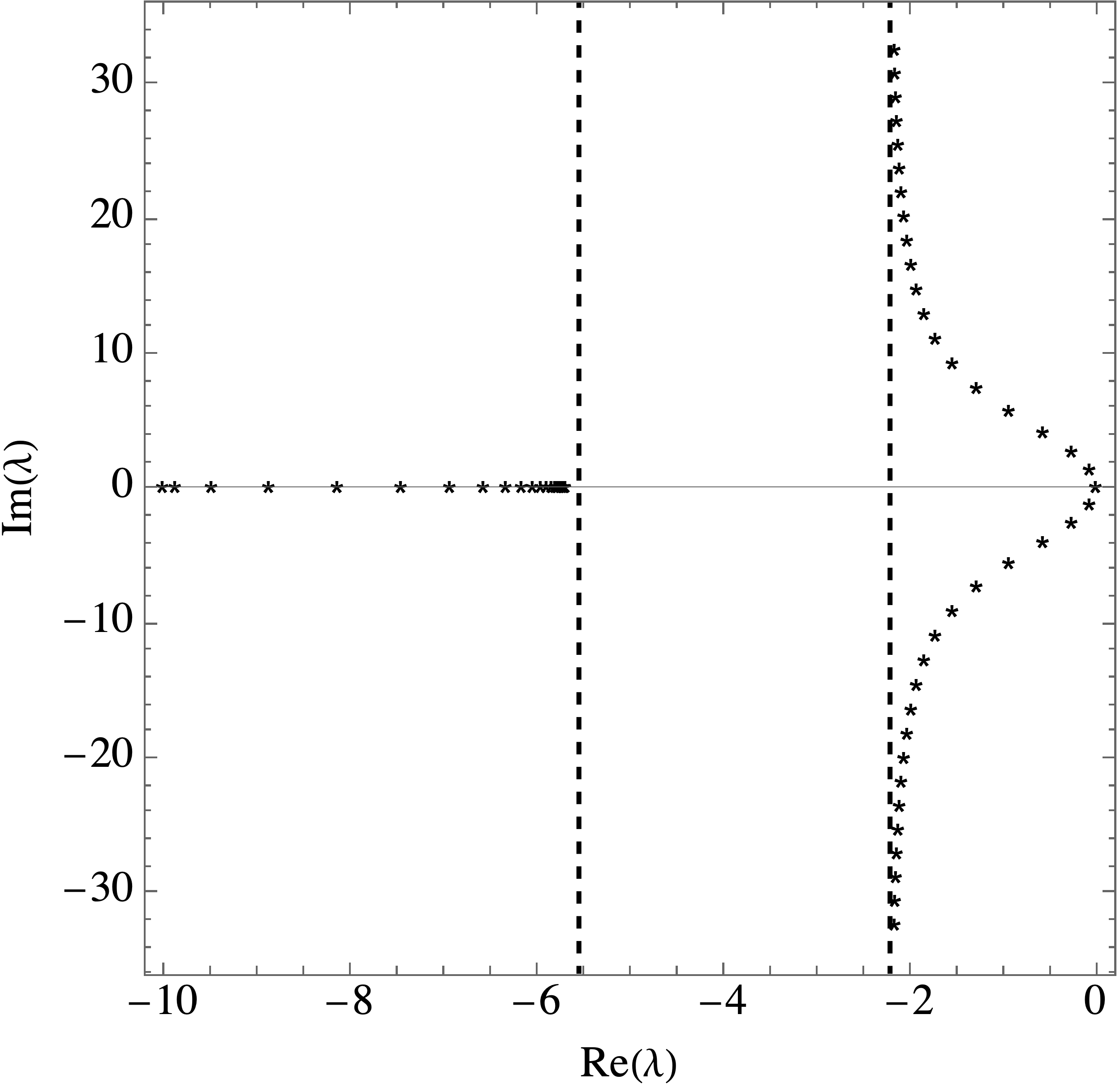}
    \caption{
    Spectrum of the  linear Grad system \eqref{eqGrad} for $\varepsilon=0.1$ and various wave numbers. The fast real modes (star symbol on the real axis) accumulate at $-\frac{5}{9\varepsilon}$ (vertical dashed line at $\Re{\lambda} = -50/9$), while the slow complex conjugated modes (bell-shaped star symbol) accumulate at $\Re\lambda = -\frac{2}{9\varepsilon}$ (dashed line at $\Re{\lambda} = -20/9$) as $k\to\infty$.}
    \label{pspec_Grad3D}
\end{figure}

Consider the three-component Grad's system,
\begin{equation}\label{eqGrad}
\begin{split}
&\frac{\partial p}{\partial t}=-\frac{5}{3}\frac{\partial u}{\partial x},\\
&\frac{\partial u}{\partial t}=-\frac{\partial p}{\partial x}-\frac{\partial \sigma}{\partial x},\\
&\frac{\partial \sigma}{\partial t}=-\frac{4}{3}\frac{\partial u}{\partial x}-\frac{1}{\varepsilon}\sigma,
\end{split}
\end{equation}
for the pressure $p$, the velocity $u$ and the stress $\sigma$, while $\epsilon>0$ is a parameter representative of the Knudsen number. System \eqref{eqGrad} is a classical model in kinetic theory \cite{grad1949kinetic} whose Chapman--Enskog series has been summed explicitly in \cite{gorban1996short}. Equation \eqref{eqGrad} is defined on the whole real line with suitable decay at infinity. In the frequency space, system \eqref{eqGrad} reads,
\begin{equation}
\frac{d}{d t}\left(\begin{array}{c}
     \hat{p}\\
     \hat{u}\\
     \hat{\sigma}
\end{array}\right) = L_k\left(\begin{array}{c}
     \hat{p}\\
     \hat{u}\\
     \hat{\sigma}
\end{array}\right)
\end{equation}

for the wave-number dependent matrix,
\begin{equation}\label{Lk}
    L_k = \begin{pmatrix}
        0 &  -\frac{5}{3}\ri k & 0\\
-\ri k & 0 & -\ri k\\
0 & -\frac{4}{3}\ri k & -\frac{1}{\varepsilon}
    \end{pmatrix},
\end{equation}
where $\ri=\sqrt{-1}$. We analyze the system \eqref{Lk} on the Hilbert space $H = \mathbb{C}^3$ with the complex inner product $\langle {v},{w}\rangle = {v}\cdot {w}^*$ for ${v},{w}\in\mathbb{C}^3$. The matrix $L_k$ is not normal since
\begin{equation}
    [L_k,L_k^\dagger] = \frac{1}{9\varepsilon}\left(
\begin{array}{ccc}
 16 k^2 \epsilon  & 0 & 11 k^2 \varepsilon  \\
 0 & -23 k^2 \epsilon  & 21 \ri k \\
 11 k^2 \epsilon  & -21 \ri k & 7 k^2 \varepsilon  \\
\end{array}
\right). 
\end{equation}
For a detailed spectral analysis of \eqref{Lk} along with its implications for the exact summation of the Chapman--Enskog series \cite{karlin2002hydrodynamics}, we refer to \cite{kogelbauer2020slow}. The characteristic polynomial of \eqref{Lk},
\begin{equation}\label{charLk}
P_{{k}}(\lambda)=-\lambda^3-\frac{1}{\varepsilon}\lambda^2-3k^2\lambda-\frac{5}{3}\frac{k^2}{\varepsilon},
\end{equation}
has a pair of complex conjugated eigenvalues $\{\lambda_{\rm  ac},\lambda_{\rm  ac}^*\}$, called acoustic modes, and one real eigenvalue $\lambda_{\rm diff}$, called generalized diffusion mode. All eigenvalues have negative real part, see Figure \ref{pspec_Grad3D}. In particular, $\text{spec}(L_k)=\text{spec}(L_k^\dagger)$. The eigenvectors of $L_k$ are given by 
\begin{equation}\label{defQ}
Q = \left(\begin{matrix}
-1-a_1b_1 & -1-a_2b_2 & -1-a_3b_3\\
\ri b_1 & \ri b_2 & \ri b_3\\
1 & 1 & 1
\end{matrix}\right),
\end{equation}
for 
\begin{equation}
a_j=\frac{\lambda_j}{k},\quad b_j=\frac{3}{4\varepsilon k}(1+\varepsilon \lambda_j),\quad j=1,2,3, 
\end{equation}
where $\lambda_1 = \lambda_{\rm ac}$, $\lambda_2 = \lambda_{\rm ac}^*$ and $\lambda_{3} = \lambda_{\rm diff}$. The slow manifold at each wave number is then given explicitly as the plane spanned by the two acoustic modes (the first two columns of \eqref{defQ}),
\begin{equation}
    \mathcal{M}_{\rm slow} = \text{span} \left\{\left(\begin{array}{c}
         -1-a_1b_1 \\
         \ri b_1 \\
         1\end{array}\right),\left(\begin{array}{c}
         -1-a_2b_2 \\
         \ri b_2 \\
         1
    \end{array}\right)\right\}.
\end{equation}
The reduced dynamics on the slow manifold in spectral coordinates \eqref{reddyn} is written,
\begin{equation}
    \frac{d\xi_{\rm ac}}{dt} = \lambda_{\rm ac} \xi_{\rm ac},\ \frac{d\xi_{\rm ac}^*}{dt} = \lambda_{\rm ac}^*\xi_{\rm ac}^*,
    \end{equation}
while the coordinate transform from spectral coordinates to the $(p,u)$-coordinates takes the form,
\begin{equation}
    H = \begin{pmatrix}
        -1 -  a_1b_1 & -1-a_2b-2 \\
        \ri b_1  & \ri b_2,
    \end{pmatrix}. 
\end{equation}
corresponding to the $p$- and $u$-components of the acoustic eigenvectors and its conjugate. 
The reduced model for the pressure and the velocity field at each wave number is then given by,
\begin{equation}\label{eq:Grad10CE}
    \frac{d}{dt}\left(\begin{array}{c}
         \hat{p}_k \\
         \hat{u}_k 
    \end{array}\right) = {T} \left(\begin{array}{c}
         \hat{p}_k \\
         \hat{u}_k 
    \end{array}\right),
\end{equation}
for the transport matrix,
\begin{equation}
    T = H^{-1}\Lambda H,\quad \Lambda = \begin{pmatrix}
        \lambda_{\rm ac} & 0 \\
         0 & \lambda_{\rm ac}^*
    \end{pmatrix}. 
\end{equation}
The system \eqref{eq:Grad10CE} was earlier derived by exact summation of the Chapman--Enskog expansion in \cite{gorban1996short}.
We also refer to \cite{kogelbauer2024rigorous} for a general treatment of the hydrodynamic closure for kinetic systems, including details on the coordinate change from spectral to physical coordinates and its relation to the Riesz projection. 
The resolvent at the negative eigenvalues can be readily calculated and takes the explicit form,
\begin{widetext}
\begin{equation}\label{resolventGrad}
    (L_k+\lambda^*)^{-1} = \frac{1}{6(\lambda^*)^2+10\varepsilon^2}\begin{pmatrix}
        -4\varepsilon k^2 + 4\lambda^* - 3\varepsilon(\lambda^*)^2 & 5\ri k - 3\ri k \lambda^* & 5\varepsilon k^2\\
        3\ri k - 3\ri \varepsilon\lambda^* & 3\lambda^* - 3\varepsilon (\lambda^*)^2 & -3\ri \varepsilon k \lambda^* \\
        4\varepsilon k^2 & -4\ri \varepsilon k \lambda^* & -5\varepsilon k^2 - 3\varepsilon (\lambda^*)^2 
    \end{pmatrix},
\end{equation}
\end{widetext}
for $\lambda\in \{\lambda_{ \rm ac},\lambda_{\rm ac}^*\}$, while the Gramian \eqref{GX} reads,
\begin{widetext}
\begin{equation}\label{GGrad}
    G = \begin{pmatrix}
        \frac{1}{2\lambda_{\rm ac}}(|1+a_1b_1|^2+|b_1|^2+1) & \frac{1}{2\Re\lambda_{\rm ac}} [(1+a_1b_1)^2+b_1^2+1] \\
        \frac{1}{2\Re\lambda_{\rm ac}} [(1+a_2b_2)^2+b_2^2+1] & \frac{1}{2\lambda_{\rm ac}^*}(|1+a_2b_2|^2+|b_2|^2+1)
    \end{pmatrix},
\end{equation}
\end{widetext}
where we have used that $a_1^*=a_2$ and $b_1^*=b_2$.
Combining \eqref{resolventGrad} with \eqref{GGrad} according to formula \eqref{defPdyn} then gives the dynamic projector for the three-component Grad system. 

Figure \ref{comp_proj_Grad3D} shows a comparison of pressure trajectories for the full system, the orthogonally projected initial condition and  the dynamically projected initial condition. As expected, the dynamically projected pressure trajectory follows the transient oscillations of the full solution much more closely. 
\begin{figure}[t]
    \centering
    \includegraphics[width=0.9\linewidth]{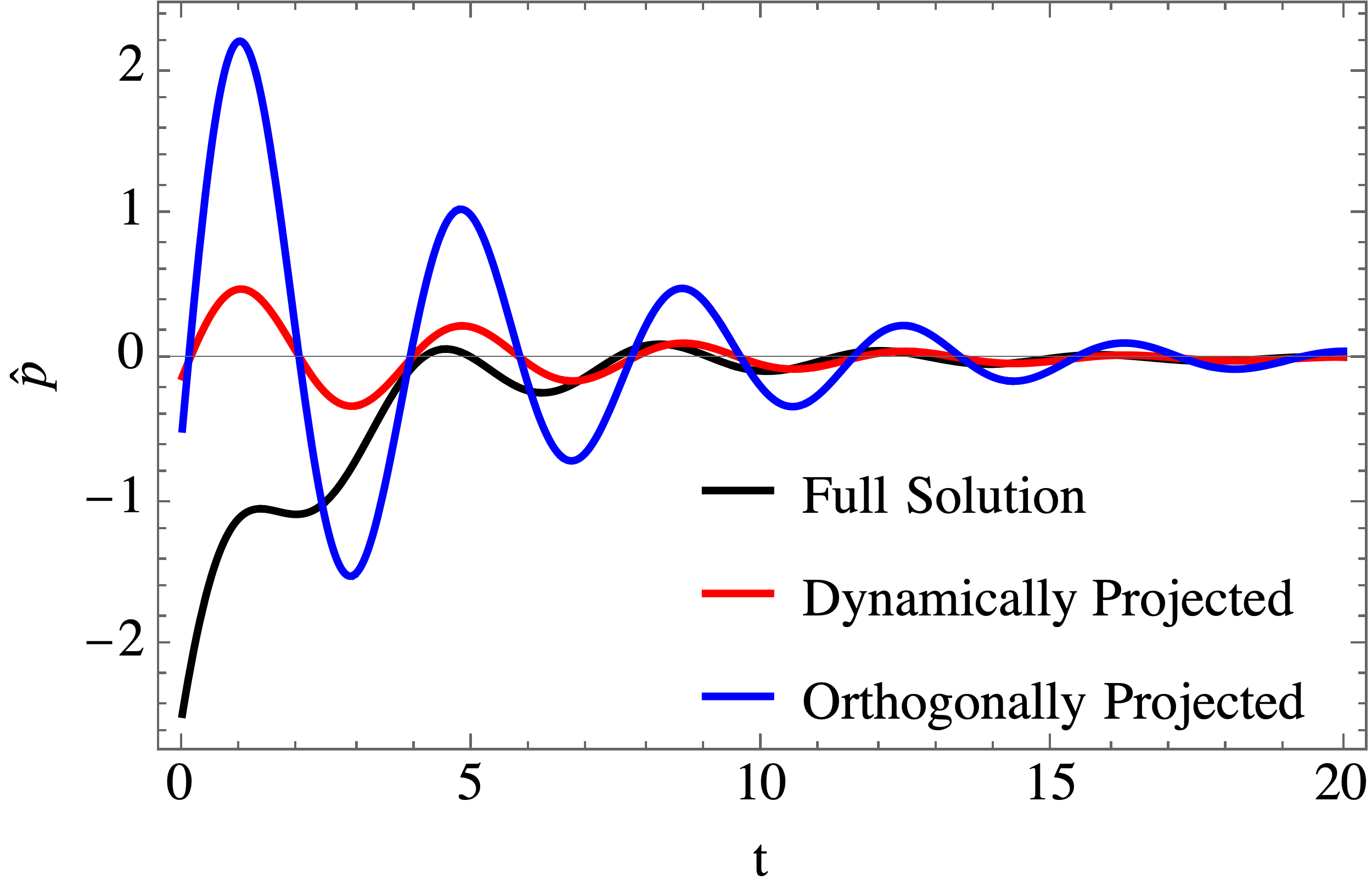}
    \caption{Comparison of the pressure components of the full solution (black) to the dynamically projected trajectory (red) and the orthogonally projected trajectory (blue) of system \eqref{eqGrad} with initial condition orthogonal to the slow manifold at wave number $k=1$. Even though the difference in initial pressure is comparably small, the orthogonally projected pressure trajectory exhibits much larger oscillations compared to the dynamically projected one.}
    \label{comp_proj_Grad3D}
\end{figure}

\section{Discussion}

We introduced a novel method to optimally project dynamics onto a slow manifold in linear systems using a temporal variational principle. Unlike traditional projection methods, which may fail to capture transient behavior, the presented approach minimizes the cumulative deviation between the full system's trajectory and its reduced counterpart over time. This method accounts for transient dynamics, particularly in systems governed by non-normal operators where transients can be significant. The projection reduces to the familiar orthogonal projection in normal systems, ensuring consistency, and it offers a mathematically explicit formulation based on spectral properties of the system's operator.

We demonstrated the effectiveness of the dynamically optimal projection through explicit examples, including a two-dimensional shear flow and the linear three-component Grad's system from kinetic theory. These case studies highlight how the new projection captures transient behaviors more accurately than standard projections, especially when dealing with systems exhibiting strong non-normality. Beyond its theoretical development, the projection has practical applications in reduced-order modeling, data assimilation, and the analysis of complex dynamical systems where transient dynamics play a critical role. Future work is anticipated to extend these ideas to non-invariant constraints and broader nonlinear contexts.

Potential applications of the dynamically optimal projection include problems where information about an entire subset of the full phase space is needed, such as sampling over a set of initial conditions or studying the evolution of specific regions of phase space. This is particularly relevant for optimal reduction on non-invariant submanifolds. A forthcoming paper will explore the application of the DOP to non-invariant constraints and their approximate reduced dynamics in greater detail.

A particularly relevant example of non-normal operators satisfying the assumptions in \eqref{propsigma} are kinetic Boltzmann-type operators,
\begin{equation}\label{LkBoltz}
    \mathcal{L}_{\bm{k}} = -\ri \bm{k}\cdot\bm{v}+L_{Q},
\end{equation}
where $L_{Q}^\dagger = L_{Q}$ is the linear collision operator and $\bm{k}$ is wave vector. Spectral properties of operators \eqref{LkBoltz} 
are discussed in, e.g., \cite{ellis1975first}. We stress that \eqref{LkBoltz} is non-normal and the degree of non-normality scales with wave number, $  [\mathcal{L}_{\bm{k}}, \mathcal{L}_{\bm{k}}^\dagger] = \mathcal{O}(|\bm{k}|)$. The three-component Grad systems discussed  in Section \ref{GradExpl} is a toy model for kinetic operators of the form \eqref{LkBoltz}, mimicking some of its pertinent properties. Recently, detailed  spectral information for the various linear kinetic operators, including the 
linear BGK  and Shakhov kinetic models have been derived \cite{kogelbauer2024exact,kogelbauer2024spectral}. These can serve as a basis for the application of the theory presented in this work.

\begin{acknowledgments}
This work was supported by the European Research Council (ERC) Advanced Grant  834763-PonD. 
Computational resources at the Swiss National Super  Computing  Center  CSCS  were  provided  under the grant s1286.
\end{acknowledgments}

\bibliographystyle{abbrv}
\bibliography{projection_bib}

\clearpage

\clearpage 

\appendix


\section{Properties of the quadratic form $\mathcal{E}$ \eqref{E}}
\label{propertiesE}

In this section, we expand and simplify the temporal integral expressions in \eqref{E} and prove the existence of a unique minimizer. We stress that the integral expression in \eqref{E} exists because of the stability assumption of the spectrum of $L$ and the resolvent estimate for the corresponding semi-group. Indeed, we recall that an eventually norm-continuous semigroup is exponentially stable if and only if $\Re\text{spec}(L)< 0$ , see \cite{engel2000one}.
For two vectors $\xi,\eta\in\mathbb{C}^n$ with $\xi=(\xi_1,...,\xi_n)$ and $\eta = (\eta_1,...,\eta_n)$, we denote the real standard inner product as 
\begin{equation}\label{eq:product}
    \xi\cdot\eta = \sum_{j=1}^n \xi_j \eta_j . 
\end{equation}
The cumulative error contains two relevant contributions. The first one corresponds to the internal dynamics on the slow manifold,
\begin{equation}\label{Einter}
\begin{split}
        \mathcal{E}_{\rm inter}(\xi ) & =  \frac{1}{2} \int_0^\infty \|e^{tL}x_{\rm slow}(\xi )\|^2\, dt,
\end{split}
\end{equation}
while the second contribution, given by an integrated inner product,
\begin{equation}\label{Etrans}
\begin{split}
        \mathcal{E}_{\rm trans}(x_0,\xi )  = -\Re\int_0^\infty \langle e^{tL}x_0,e^{tL}x_{\rm slow}(\xi )\rangle\, dt,
\end{split}
\end{equation}
corresponds to interaction of the dynamics on the slow manifold with a general trajectory of \eqref{main}. The full functional \eqref{E} is then given by,
\begin{equation}\label{Esum}
    \mathcal{E}(x_0,\xi ) = \mathcal{E}_{\rm inter}(\xi ) + \mathcal{E}_{\rm trans}(x_0,\xi ) + \frac{1}{2} \int_0^\infty \|e^{tL}x_{0}\|^2\, dt. 
\end{equation}
Since the last summand in \eqref{Esum} does not depend on $\xi $, the minimizer of $\mathcal{E}$ only depends on $\mathcal{E}_{\rm inter}$ and $\mathcal{E}_{\rm trans}$.

First, let us simplify the expressions \eqref{Einter} and \eqref{Etrans} by resolving the temporal integrals.
The internal dynamical contribution can be expanded as,
\begin{equation}\label{expandEinter}
\begin{split}
        \mathcal{E}_{\rm inter}(\xi ) & =  \frac{1}{2} \int_0^\infty \|e^{tL}x_{\rm slow}(\xi )\|^2\, dt\\
        & = \frac{1}{2} \sum_{j,k=1}^n \int_0^\infty \langle\xi_j e^{tL}\hat{x}_j,\xi_k e^{tL}\hat{x}_k\rangle\, dt\\
        & = \frac{1}{2} \sum_{j,k=1}^n \int_0^\infty \langle\hat{x}_j,\hat{x}_k\rangle \xi_j\xi_k^*e^{(\lambda_j+\lambda_k^*)t}\, dt\\
        & =  \frac{1}{2} \sum_{j,k=1}^n  \langle\hat{x}_j,\hat{x}_k\rangle \xi_j\xi_k^*\left[ \frac{e^{(\lambda_j+\lambda_k^*)t}}{\lambda_j+\lambda_k^*}\right]_{0}^{\infty} \\
        & = - \frac{1}{2} \sum_{j,k=1}^n\xi_j \left[\frac{\langle\hat{x}_j,\hat{x}_k\rangle }{\lambda_j+\lambda_k^*}\right]\xi_k^*\\
        & = -\frac{1}{2} \xi\cdot G\xi ^*,
\end{split}
\end{equation}
where
\begin{equation}\label{GXapp}
    G_{jk} = \frac{\langle\hat{x}_j ,\hat{x}_k\rangle}{\lambda_j+\lambda_k^*},
\end{equation}
is a spectrally-weighted Gramian matrix associated to the eigenvectors of $L$ as defined in \eqref{GX}. Note that we have only used the assumption $\Re\text{spec}(L)<0$ in the above calculations.

To evaluate the dynamical interaction contribution, we expand \eqref{Etrans} as 
\begin{equation}\label{expandEtrans}
\begin{split}
        \mathcal{E}_{\rm trans}(x_0,\xi ) & = -\Re\int_0^\infty \langle e^{tL}x_0,e^{tL}x_{\rm slow}(\xi )\rangle\, dt\\
         & = -\sum_{j=1}^n\Re\int_0^\infty \langle e^{tL}x_0,e^{\lambda_j t}\hat{x}_j\rangle \xi_j^*\, dt\\
         & = -\sum_{j=1}^n\Re\int_0^\infty \langle x_0, e^{\lambda_j t}e^{tL^\dagger}\hat{x}_j\rangle \xi_j^*\, dt\\
          & = -\sum_{j=1}^n\Re[\langle x_0, \eta_j \rangle\xi_j^*], 
\end{split}
\end{equation}
where we have defined the vector $\eta$ with the components,
\begin{equation}\label{eq:ethai}
    \eta_j =  \int_0^\infty e^{tL^\dagger}\hat{x}_j e^{\lambda_j t}\, dt.
\end{equation}
Integration by parts in \eqref{eq:ethai} gives,
\begin{equation}\label{eqeta}
    \begin{split}
    \eta_j & =  \int_0^\infty e^{tL^\dagger}\hat{x}_j e^{\lambda_j t}\, dt \\
    & = \frac{1}{\lambda_j}\left[ e^{tL^\dagger}\hat{x}_j e^{\lambda_j t} \right]_0^\infty - \frac{1}{\lambda_j}\int_0^\infty L^\dagger e^{tL^\dagger}\hat{x}_j e^{\lambda_j t}\, dt\\
     & = -\frac{\hat{x}_j}{\lambda_j} - \frac{1}{\lambda_j } L^\dagger \eta_j,
    \end{split}
\end{equation}
where we have again used that $\Re\text{spec}(L)<0$.  From \eqref{eqeta}, we infer the following linear equation for $\eta_j$:
\begin{equation}\label{eqeta2}
    (L^\dagger+\lambda_j)\eta_j = -\hat{x}_j. 
\end{equation}
Because of the stability assumption $\Re \text{spec}(L)<0$ and since $\text{spec}(L^\dagger)=\text{spec}(L)^*$, the complex number $-\lambda_j$ does not belong to the spectrum of $L^\dagger$. Hence, we can solve equation \eqref{eqeta2} to give,
\begin{equation}
    \eta_j = - (L^\dagger+\lambda_j)^{-1}\hat{x}_j.  
\end{equation}
Thus, expression \eqref{Etrans} consequently simplifies as,
\begin{equation}\label{Etranssimple}
\begin{split}
        \mathcal{E}_{\rm trans}(x_0,\xi ) & = \sum_{j=1}^n\Re[\langle x_0,  (L^\dagger+\lambda_j)^{-1}\hat{x}_j \rangle\xi_j^*]\\
        & = \sum_{j=1}^n\Re[\langle (L+\lambda_j^*)^{-1} x_0,  \hat{x}_j \rangle\xi_j^*]\\
        & = \Re[I \cdot \xi ^*],
\end{split}
\end{equation}
where we have defined a vector $I(x_0)$ with the components,
\begin{equation}\label{defI}
I_j(x_0) =  \langle (L+\lambda_j^*)^{-1} x_0,  \hat{x}_j \rangle,\ {1\leq j \leq n}. 
\end{equation}
Clearly, from  \eqref{expandEinter} and \eqref{Etranssimple}, the function $\xi \mapsto \mathcal{E}(x_0,\xi )$ defines a quadratic form for any $x_0\in H$. To see that $\mathcal{E}$ is a convex function in $\xi $, we observe that the second-order term of $\mathcal{E}$ is given by \eqref{Einter}, which is strictly positive as the integral of a norm for any $\xi \in\mathbb{C}^n$.  In particular, the spectrally-weighted Gramian matrix $G$ is negative semi-definite, as it follows from \eqref{expandEinter}.
Hence the quadratic part of $\mathcal{E}$ is positive definite and $\mathcal{E}$ is convex.\\

\section{Explicit expression of the minimizer}
\label{minexplicit}

In this section, we derive an explicit expression of the unique minimizer of the quadratic form $\xi\mapsto\mathcal{E}(x_0,\xi)$ on the slow manifold. To evaluate the complex derivative $\partial \mathcal{E}/ \partial \xi $, we recall the following lemma (proof is included for completeness):

\begin{lemma}\label{derlemma}
    Consider a quadratic form,
\begin{equation}\label{deff}
    f(\xi ) = \Re[w\cdot\xi ^*]- \frac{1}{2}\xi \cdot G\xi ^*,
\end{equation}
for a Hermitian matrix $G=G^\dagger$ and complex vectors $w=w_R+\ri w_I$,  $\xi =\xi _R+\ri\xi _I$. Then the complex derivative of $f$ is given by
\begin{equation}
    \frac{\partial f}{\partial \xi } = w - G^T\xi . 
\end{equation}
\end{lemma}
\begin{proof}
First, we expand 
\begin{equation}
    \begin{split}
        \Re[w\cdot\xi ^*] & = w_R\cdot\xi _R+w_I\cdot \xi_I,\\
        \xi \cdot G\xi ^* & = \xi _R\cdot G\xi _R+\xi _I\cdot  G\xi _I\\
        &\qquad -\ri\xi _R\cdot G\xi _I+\ri\xi _I\cdot G\xi _R,
    \end{split}
\end{equation}
and set
\begin{equation}
    G_S = \frac{1}{2}(G+G^T),\ G_A =\frac{1}{2}(G-G^T),
\end{equation}
to obtain,
\begin{equation}
    \begin{split}
        \frac{\partial f}{\partial\xi _R} & = w_R - G_S\xi _R+\ri G_A\xi _I,\\
         \frac{\partial f}{\partial\xi _I} & = w_I - G_S\xi _I-\ri G_A\xi _R.
    \end{split}
\end{equation}
Consequently, the complex derivative of \eqref{deff}
is given by
\begin{equation}
\begin{split}
        \frac{\partial f}{\partial \xi } & = \frac{\partial f}{\partial\xi _R} + \ri  \frac{\partial f}{\partial\xi _I}\\
         & = (w_R - G_S\xi _R+\ri G_A\xi_I) + \ri (w_I - G_S\xi _I-\ri G_A\xi _R)\\
          & = w- G^T\xi _R-\ri G^T\xi _{I}\\
          & = w- G^T\xi. 
\end{split}
\end{equation}
\end{proof}

Applying Lemma \ref{derlemma} to $w = I (x_0)$ and $G$ as defined in \eqref{GXapp}, it follows that the critical point of the function \eqref{E} satisfies the equation 
\begin{equation}\label{derivative}
    I(x_0)-G^T\xi = 0.
\end{equation}
As shown before, the critical point indeed defines a minimum thanks to the convexity of $\mathcal{E}$. The unique minimizer of \eqref{E} is then given explicitly as the solution to \eqref{derivative}:
\begin{equation}\label{betaminapp}
    \xi^{\rm min}(x_0) =  (G^{T})^{-1} I (x_0).
\end{equation}

\section{Projection property of DOP \eqref{defPdyn}}
\label{projpropapp}

Let us prove that parameter value that minimizes the cumulative dynamical error,
\begin{equation}
    \xi^{\rm min}(x_0) := \argmin_{\xi \in\mathbb{C}^n} \mathcal{E}(x_0,\xi ),
\end{equation}
indeed, defines a projection onto the slow manifold as
\begin{equation}\label{defPslowapp}
\mathbb{P}_{\rm DOP} x =  \sum_{i=1}^n\sum_{j=1}^n\hat{x}_i\left[\left(G^{T}\right)^{-1}\right]_{ij} \langle (L+\lambda_j^*)^{-1}x,\hat{x}_j \rangle.
\end{equation}
To ease notation, we bundle the eigenvectors into a vector of Hilbert space elements, $\hat{x} = (\hat{x}_1,...,\hat{x}_n)$; elements of the slow manifold \eqref{slowmf} are written more compactly using the real standard inner product \eqref{eq:product},
\begin{equation}
    \label{eq:para_slow}
    x_{\rm slow} = \xi\cdot\hat{x}.
\end{equation}
First, we evaluate the vector \eqref{defI} on the elements of the slow manifold \eqref{eq:para_slow},
%
%
\begin{equation}\label{projI}
\begin{split}
       I_j \left(\xi\cdot\hat{x}\right) & 
        = \sum_{k=1}^{n} \xi_k\langle (L+\lambda_j^*)^{-1}\hat{x}_k, \hat{x}_j\rangle \\
       & = \sum_{k=1}^{n}  \left[\frac{\langle\hat{x}_k,\hat{x}_j\rangle}{\lambda_k+\lambda_j^*}\right]\xi_k\\
       & = [G^T\xi ]_j.
\end{split}
\end{equation}
This can be written in a form of an identity,
\begin{equation}
    \label{eq:para-identity}
    \xi=\left[G^T\right]^{-1}I\left(\xi\cdot\hat{x}\right),
\end{equation}
for any $\xi \in\mathbb{C}^n$, while operator \eqref{defPslowapp} takes the form,
\begin{equation}
    \label{eq:proj_simp}
    \mathbb{P}_{\rm DOP} x =\hat{x}\cdot\left[G^T\right]^{-1}I\left(x\right).
\end{equation}
Second, applying the operator \eqref{eq:proj_simp} twice and using \eqref{eq:para-identity} for $\xi  =  \left[G^{T}\right]^{-1} I (x)$, we find, 
\begin{equation}
\begin{split}
        \mathbb{P}^2_{\rm dyn}x & =  \mathbb{P}_{\rm DOP}\hat{x}\cdot \left[G^{T}\right]^{-1} I (x)\\
        & = \hat{x}\cdot \left[G^{T}\right]^{-1} I \left(\hat{x}\cdot \left[G^{T}\right]^{-1} I (x)\right)\\
        & = \hat{x}\cdot \left[G^{T}\right]^{-1} G^{T} \left[G^{T}\right]^{-1} I (x)\\
        & = \hat{x}\cdot \left[G^{T}\right]^{-1} I (x)\\
        & = \mathbb{P}_{\rm DOP}x.
\end{split}
\end{equation}
This proves Proposition \ref{prop:projector}.

\section{DOP reduces to the canonical orthogonal projection for $L$ normal}\label{Lnormal}

In this section, we assume that operator $L$ is normal, $[L,L^\dagger]=0$, and recall that eigenvectors of a normal operator corresponding to different eigenvalues are orthogonal. For the sake of presentation, we also assume that the eigenvalues $\lambda_j$ are simple; thus, $\langle \hat{x}_i,\hat{x}_j\rangle=\delta_{ij}$. The case of repeated eigenvalues and generalized eigenvectors leads to Jordan blocks in the matrix $G$ and can be treated analogously. 

As an immediate consequence, the spectrally-weighted Gramian matrix \eqref{GX} becomes diagonal,
\begin{equation}\label{eq:Gdiag}
   G_{ij}=\frac{1}{2\Re\lambda_i}\delta_{ij},
\end{equation}
where $\delta_{ij}$ is Kronecker's delta.
Let us recall that, if $\hat{x}_j$ is unit-length eigenvector of $L$ with eigenvalue $\lambda_j$, then it is also the eigenvector of the adjoint $L^\dagger$ with the eigenvalue $\lambda_j^*$,
\begin{equation}
    L^\dagger \hat{x}_j = \lambda_j^* \hat{x}_j. 
\end{equation}
Consequently, we find in \eqref{defI},
\begin{equation}\label{Inormal}
\begin{split}
        I_j(x) & = \langle x, (L^\dagger+\lambda_j)^{-1} \hat{x}_j \rangle\\
        & = \langle x, (\lambda_j^*+\lambda_j)^{-1} \hat{x}_j \rangle \\
        & = \frac{\langle x, \hat{x}_j \rangle}{\lambda_j+\lambda_j^*}\\
        & = \frac{\langle x, \hat{x}_j \rangle}{2\Re\lambda_j}.
\end{split}
\end{equation}
With \eqref{eq:Gdiag} and \eqref{Inormal}, we obtain in \eqref{defPdyn},
\begin{equation}\label{eq:Pdyn_normal}
\begin{split}
      \mathbb{P}_{\rm DOP} x & = 
\sum_{i=1}^n\sum_{j=1}^n\hat{x}_i\left(2\Re \lambda_i\delta_{ij}\right)\frac{\langle x, \hat{x}_j \rangle }{2\Re\lambda_j}\\
       & = \sum_{j=1}^n \langle x,\hat{x}_j\rangle \hat{x}_j.
\end{split}
\end{equation}
Thus, the dynamically optimal projection becomes the standard orthogonal projection if operator $L$ is normal.

\section{Characterization of the Riesz projection through commutativity}\label{charRiesz}

In this section, we discuss properties of the Riesz projection onto a finite set of eigenvectors associated to isolated finite-order eigenvalues of a closed operator $L$. For completeness, we will show that the Riesz projection can be characterized through the commutation property with $L$. More specifically, we show that if a projection onto a finite set of linearly independent vectors commutes with $L$, then the projection is defined by eigenvectors of the adjoint of $L$.

Let $\hat{x}_1,...,\hat{x}_n$ be linearly independent unit-length eigenvectors of a closed operator,
\begin{equation}\label{eq:eigL}
    L \hat{x}_j  = \lambda_j \hat{x}_j .
\end{equation}
For any compact subset of the spectrum, $\Lambda\subseteq \text{spec}(L)$, consisting only of isolated eigenvalues of finite multiplicity, the Riesz projection is defined as  
\begin{equation}\label{defRiesz}
    \mathbb{P}_\Lambda = -\frac{1}{2\pi \ri}\oint_{\Gamma (\Lambda)} (z-L)^{-1} \, dz,
\end{equation}
where $\Gamma(\Lambda)$ is any simple contour encircling the set $\Lambda$ once in the positive direction and lying entirely in the resolvent set of $L$. The linear operator \eqref{defRiesz} indeed defines a projection \cite{hislop2012introduction}, its range is exactly the linear subspace spanned by the eigenvectors associated to $\Lambda$. Moreover, the Riesz projection \eqref{defRiesz} commutes with the operator $L$,
\begin{equation}
    \label{eq:Riesz_commutator}
    [\mathbb{P}_\Lambda, L]=0.
\end{equation}

On the other hand, a general linear projection onto the subspace spanned by $\{\hat{x}_1,...,\hat{x}_n\}$ is defined through a linear mapping ${M}\in L(H,\mathbb{C}^n)$ such that
\begin{equation}\label{projM}
    \mathbb{P}_{M}x = \sum_{j=1}^n \hat{x}_j  ({M}x)_j.
\end{equation}
By the Riesz representation theorem \cite{conway2019course}, there exist vectors $\theta_1,..,\theta_n\in H$ such that
\begin{equation}\label{representation}
    \mathbb{P}_{M}x = \sum_{j=1}^n  \hat{x}_j  \langle x,\theta_j\rangle. 
\end{equation}
The linear mapping \eqref{projM} defines a projection (involution) provided that
\begin{equation}\label{restrict}
\langle {\hat{x}_i}, \theta_j\rangle =\delta_{ij},
\end{equation}
i.e., if $\{\theta_{j}\}$ is a dual set of $\{\hat{x}_j \}$. Moreover, the following commutation property characterizes the projection uniquely:

\begin{lemma}
    Let $\mathbb{P}$ be a linear projection onto the linearly independent eigenvectors $\hat{x}_1,...,\hat{x}_n$ of a closed operator $L$. Assume that $\mathbb{P}$
 commutes with $L$. Then the vectors $\theta_1,..,\theta_n$ in the representation \eqref{representation} are eigenvectors of the adjoint of $L$, 
 \begin{equation}\label{theta}
     L^\dagger\theta_j = \lambda_ j^* \theta_j. 
 \end{equation}
 \end{lemma}
 \begin{proof}
From 
 \begin{equation}
 \begin{split}
     \mathbb{P}L x &= \sum_{j=1}^n \langle Lx,\theta_j \rangle \hat{x}_j \\
     & = \sum_{j=1}^n \langle x,L^\dagger\theta_j \rangle \hat{x}_j ,
\end{split}
 \end{equation}
 and 
 \begin{equation}
 \begin{split}
          L\mathbb{P}x & = \sum_{j=1}^n \langle x,\theta_j \rangle L\hat{x}_j \\
        & = \sum_{j=1}^n \lambda_j\langle x,\theta_j \rangle \hat{x}_j ,
 \end{split}
 \end{equation}
 where we have used that $\hat{x}_j$'s are eigenvectors of $L$ \eqref{eq:eigL}, it follows from the linear independence of the $\hat{x}_j$'s that
 \begin{equation} 
     \langle x,L^\dagger \theta_j \rangle = \lambda_j \langle x, \theta_j\rangle =\langle x,\lambda_j^*\theta_j\rangle. 
 \end{equation}
 Since $x$ is arbitrary, the statement \eqref{theta} follows. 
 \end{proof}
 Thus, the commutation property \eqref{eq:Riesz_commutator} defines the Riesz projection uniquely as the projection with the null-space 
 $\ker \mathbb{P}=\{x\in H:x\perp\theta_j,\ L^\dagger\theta_j=\lambda^*_j\theta_j,\ 1\le j\le  n\}$ and range ${\rm im}\, \mathbb{P}={\rm span}\{\hat{x}_j ,\ 1\le j\le n\}$.

Finally, we note that the Riesz representation \eqref{representation} of the DOP \eqref{defPdyn} corresponds to the following dual set,
 \begin{equation}
     \label{eq:dynProj_R}
     \theta_i=\sum_{j=1}^n\left[\left(G^{T}\right)^{-1}\right]_{ij} (L^\dagger+\lambda_j)^{-1}\hat{x}_j,\ 1\le i\le n.
 \end{equation}
 Indeed,
 \begin{equation}
 \begin{split}
          \langle\hat{x}_k,\theta_i\rangle & = \sum_{j=1}^n \left[\left(G^{T}\right)^{-1}\right]_{ij} 
          \langle \hat{x}_k, (L^\dagger+\lambda_j)^{-1}\hat{x}_j\rangle \\
        & =  \sum_{j=1}^n \left[\left(G^{T}\right)^{-1}\right]_{ij} 
          \langle (L+\lambda_j^*)^{-1}\hat{x}_k, \hat{x}_j\rangle\\
        &= \sum_{j=1}^n \left[\left(G^{T}\right)^{-1}\right]_{ij} 
         \left[ \frac{\langle \hat{x}_k, \hat{x}_j\rangle}{\lambda_k+\lambda_j^*}\right]\\
         & = \sum_{j=1}^n \left[\left(G^{T}\right)^{-1}\right]_{ij} 
          G_{kj}\\
         & = \sum_{j=1}^n \left[\left(G^{T}\right)^{-1}\right]_{ij} 
         \left[ G^T \right]_{jk}\\
         &=\delta_{ik},\ 1\leq i,k\leq n.
 \end{split}
 \end{equation}
 From the above considerations, it becomes apparent that DOP \eqref{defPdyn} is different from the Riesz projection \eqref{defRiesz} for non-normal operators $L$, while at the same time, the DOP reduces to the Riesz projection and becomes the standard orthogonal projection \eqref{eq:Pdyn_normal} when $L$ is normal.

\clearpage

\end{document}